\documentclass[11pt]{amsart}
\usepackage[T1]{fontenc}
\usepackage[utf8]{inputenc}
\usepackage[margin=1.25 in]{geometry} 

\usepackage{graphicx}
\usepackage{url}
\usepackage[pagebackref=false,colorlinks]{hyperref}
\hypersetup{citecolor={rgb,256:red,0;green,50;blue,100},
linkcolor={rgb,256:red,;green,60;blue,160}}
\usepackage{mathptmx}     
\usepackage{amsmath,amsfonts,amsthm,amssymb,color}
\usepackage{mathtools}
\usepackage{tikz-cd}
\usepackage{quiver}
    \usetikzlibrary{decorations.pathmorphing}
\usepackage{stmaryrd}
\usepackage{bbm}
\usepackage{subcaption}
\usepackage[shortlabels]{enumitem}

\let\saveamalg\amalg
\let\amalg\relax
\usepackage{mathabx}
\let\amalg\saveamalg

\usepackage{xparse}
\usepackage{calrsfs}
\usepackage{cleveref}

\DeclareMathAlphabet{\pazocal}{OMS}{zplm}{m}{n}
\usepackage{todonotes}

\newtheorem{theorem}{Theorem}[section]
\newtheorem{corollary}[theorem]{Corollary}

\newtheorem{lemma}[theorem]{Lemma}

\newtheorem{proposition}[theorem]{Proposition}

\newtheorem{athm}{Theorem}

\newtheorem{acor}{Corollary}[athm]

\theoremstyle{definition}
\newtheorem{definition}[theorem]{Definition}

\newtheorem{remark}[theorem]{Remark}
\newtheorem{example}[theorem]{Example}
\newtheorem{notation}[theorem]{Notation}

\newcommand{\cA}{\pazocal{A}} %a (commutative)(Frobenius) algebra, also regarded as a symmetric monoidal functor
\newcommand{\cL}{\pazocal{L}}
\newcommand{\cE}{\pazocal{E}}
\newcommand{\cF}{\pazocal{F}}
\newcommand{\cG}{\pazocal{G}}
\newcommand{\Z}{\mathbb{Z}}
\newcommand{\Mod}{\mathrm{Mod}}
\newcommand{\Gr}{\mathrm{Gr}}
\newcommand{\GrCob}{\mathrm{GrCob}}
\newcommand{\Fin}{\mathrm{Fin}}
\newcommand{\Disk}{\mathrm{Disk}}
\newcommand{\Cob}{\mathrm{Cob}}
\newcommand{\Mfld}{\mathrm{Mfld}}
\newcommand{\map}{\mathrm{Map}}
\newcommand{\fib}{\mathrm{fib}}
\newcommand{\E}{\mathbf{E}}
\newcommand{\V}{\pazocal{V}}
\newcommand{\colim}{\mathrm{colim}}
\newcommand{\C}{\pazocal{C}}
\newcommand{\D}{\pazocal{D}}
\newcommand{\Psh}{\mathrm{Psh}}
\newcommand{\id}{\mathrm{id}}
\newcommand{\Ss}{\pazocal{S}}
\newcommand{\too}{\longrightarrow}
\newcommand{\op}{\mathrm{op}}
\newcommand{\GFT}{\mathrm{GFT}}
\newcommand{\Cospan}{\mathrm{Cospan}}
\newcommand{\ot}{\leftarrow}
\newcommand{\Fun}{\mathrm{Fun}}
\newcommand{\cat}{\mathrm{Cat_\infty}}
\newcommand{\CAlg}{\mathrm{CAlg}}
\newcommand{\Ei}{{\E_\infty}}
\newcommand{\OC}{\pazocal{OC}}
\newcommand{\cO}{\pazocal{O}}
\newcommand{\Frob}{\mathrm{Frob}}
\newcommand{\CFrob}{\Frob_{\E_\infty}}
\newcommand{\CFrobext}{\pazocal{F}\mathrm{rob}_{\E_\infty}^\ext}
\newcommand{\Nat}{\mathrm{Nat}}
\newcommand{\Sp}{\mathrm{Sp}}
\newcommand{\ext}{\mathrm{ext}}
\newcommand{\lax}{\mathrm{lax}}
\newcommand{\Imm}{\mathrm{Im}}
\DeclareFontFamily{U}{min}{}
\DeclareFontShape{U}{min}{m}{n}{<-> udmj30}{}
\newcommand\yo{\!\text{\usefont{U}{min}{m}{n}\symbol{'207}}\!} %Yoneda embedding
\title{Universal property of graph cobordisms}
\author{Andrea Bianchi and Adela YiYu Zhang}
\date{}

\begin{document}
\begin{abstract}
We exhibit the symmetric monoidal $\infty$-category $\GrCob$ of graph cobordisms between spaces as a full $\infty$-subcategory of the $\infty$-category $\Psh(\Gr)$ of presheaves over $\Gr$, where $\Gr$ is the symmetric monoidal $\infty$-category of graph cobordisms between finite sets. We also describe a universal property for $\GrCob$: it is, in a suitable sense, the free symmetric monoidal extension of $\Gr$ endowed with the factorization homology $\int_X*$ of the universal $\Ei$-Frobenius algebra for all spaces $X$. As a corollary, we identify the space of universal natural operations on the factorization homologies, in particular the Hochschild homology, of $\Ei$-Frobenius algebras.
\end{abstract}
\maketitle

\section{Introduction}
The symmetric monoidal $\infty$-category $\GrCob$ of \emph{graph cobordisms between spaces} was introduced in \cite{Bianchi:stringtopology} as a convenient source for string topology operations. More precisely, for any $\Ei$-ring spectrum $R$ and any $R$-oriented Poincar\'e duality space $M$, the first named author costructs a ``graph field theory'', i.e.~a symmetric monoidal $R$-linear functor $\GFT_M$ from a suitable $R$-linearisation of $\GrCob^\op$ to the $R$-linear $\infty$-category $\Mod_R(\Sp)$ of $R$-modules in spectra. Evaluation at suitable morphisms in $\GrCob^\op$ recovers the most basic homotopy invariant string topology operations, in particularly the Chas--Sullivan product.

The construction of $\GrCob$ from \cite{Bianchi:stringtopology} is quite involved. One of the main ingredients is the symmetric monoidal $\infty$-category $\Gr$ of \emph{graph cobordisms between finite sets}, which can be characterised as the symmetric monoidal $\infty$-category corepresenting $\Ei$-Frobenius algebras, thanks to a recent result of Barkan--Steinebrunner \cite{BarkanSteinebrunner}; see later Theorem \ref{thm:BS}. Here an $\Ei$-Frobenius algebra in a symmetric monoidal $\infty$-category is an $\Ei$-algebra $\cA$ together with a map $\lambda:\cA\to\mathbf{1}$ to the monoidal unit, such that postcomposing with the multiplication yields a nondegenerate pairing $\cA\otimes \cA\xrightarrow{\mu}\cA\xrightarrow{\lambda}\mathbf{1}$.\footnote{This notion was introduced as a \emph{Frobenius algebra object} in \cite[4.6.5.1]{HA}.}  It turns out that $\Gr$ is equivalent to a symmetric monoidal full $\infty$-subcategory of $\GrCob$ \cite[Corollary 7.11]{Bianchi:stringtopology}; nevertheless, no ``extension'' of the universal property of $\Gr$ to a universal property of $\GrCob$ was provided, and even less employed, in loc.cit.. The graph field theory $\GFT_M$ was rather constructed by using the universal property of $\Gr$, together with a series of ad hoc and somewhat involved manipulations.

The first goal of this article is to provide a universal property for $\GrCob$ in terms of factorization homology over arbitrary spaces of the $\Ei$-Frobenius algebra corresponding to the inclusion $\Gr\to\GrCob$. The approach we take is inspired by and parallel to the work of Barkan, Steinebrunner and the second named author \cite{BSZ}, where a universal property of the open-closed 2-dimensional cobordism category $\OC$ was established in terms of Hochschild homology of $\E_1$-Frobenius algebra, and many of the proofs in this article make use of the technical results established therein.

In order to formulate this universal property, we need the following definitions.
\begin{definition}
\label{defn:intXA}
For a space $X\in\Ss$, we abbreviate by $\Fin_{/X}$ the $\infty$-category $\Fin\times_\Ss\Ss_{/X}$, where the pullback is formed along the inclusion $\Fin\hookrightarrow\Ss$ and the source functor $\Ss_{/X}\to\Ss$. For an arbitrary functor $F\colon\Fin\to\V$, we denote by $\int_XF$ the following colimit, if it exists:
\[\int_XF:=\colim_{\Fin_{/X}}F\in\V.
\] 
In the above formula the restriction of $F$ along the source functor $\Fin_{/X}\to\Fin$ is implicit. We refer to $\int_XF$ as the \emph{factorization homology over $X$ with coefficients in $F$}. 
\end{definition}
When $F\colon\Fin\to\V$ is a symmetric monoidal functor and $X$ is a framed $n$-manifold, this definition recovers the usual notion of factorization homology of an $\Ei$-algebra considered as an $\E_n$-algebra, see \cite[Proposition 5.7]{ayala2020handbook}.  In this case, the factorization homology $\int_X F$ also agrees with (the underlying object of) the tensor product ``$X \otimes F$'', i.e.~the constant colimit $\colim_XF$ computed in $\CAlg(\V)$. Most notably, when $X=S^1$, this agrees with the Hochschild homology of the underlying $\E_1$-algebra; see also \cite[IV.2.2]{NS2018} and also Lemma \ref{lem:final Disk/S^1->Fin} below.

We introduce the factorization homology $\int_XF$ for arbitrary functors out of $\Fin$ in order to ease notation later.
\begin{definition}
\label{defn:CFrobext}
We say that an $\E_\infty$-Frobenius algebra $\cA:\Gr\to\V$ in a symmetric monoidal $\infty$-category $\V$ is \textit{extendable} if the following conditions hold:
\begin{itemize}
\item For all spaces $X\in\Ss$, the factorization homology $\int_X\cA$ exists in $\V$; here we consider $\Fin$ as an $\infty$-subcategory of $\Gr$;
\item For all $X\in\Ss$ and $v\in\V$, the following assembly map is an equivalence:
\[
\int_{X}\left(\cA\otimes v\right)\xrightarrow{\simeq}\left(\int_X\cA\right)\otimes v.
\]
\end{itemize}
We introduce an $\infty$-subcategory $\CFrobext\subseteq\CAlg(\cat)_{\Gr/}$ whose objects are all extendable $\E_\infty$-Frobenius algebras $(\V,\cA)$. A morphism $F\colon(\V,\cA)\to(\V',\cA')$ in $\CAlg(\cat)_{\Gr/}$ between extendable $\Ei$-Frobenius algebras belongs to $\CFrobext$ if, for all $X\in\Ss$, the following assembly map is an equivalence:
\[
\int_X\cA'\simeq\int_XF(\cA)\to F\left(\int_X\cA\right).
\]
\end{definition}
\begin{athm}
\label{thm:A}
Consider $\GrCob$ as an object in $\CAlg(\cat)_{\Gr/}$ via the inclusion $\iota\colon\Gr\hookrightarrow\GrCob$. Then $\GrCob$ is extendable  and is an initial object in $\CFrobext$.
\end{athm}

As an immediate application of Theorem \ref{thm:A}, we obtain the following corollary, providing string operations parametrised by $\GrCob$ at the level of factorization homologies of $\Ei$-Frobenius algebras, and by restriction, at the level of Hochschild homology. 
\begin{acor}
\label{cor:A}
Let $\V$ be a cocomplete symmetric monoidal $\infty$-category whose monoidal product preserves colimits in each variable separately (e.g. $\V$ is the $\infty$-category of $R$-modules in spectra for some $\Ei$-ring spectrum $R$). Let $\cA\colon\Gr\to\V$ be an $\Ei$-Frobenius algebra. Then there is an essentially unique symmetric monoidal extension
\[
\int_{(-)}\cA\colon\GrCob\to\V
\]
whose value at any space $X\in\GrCob$ is equivalent to the factorization homology $\int_X\cA$ along the canonical map  induced by the functor $\Fin_{/X}\to\GrCob_{/X}$. Moreover the underlying functor of $\int_{(-)}\cA$ is equivalent to the left Kan extension $\iota_!\cA$ of $\cA\colon\Gr\to\V$ along $\iota\colon\Gr\hookrightarrow\GrCob$. 
\end{acor}
In a similar spirit, Barkan, Steinebrunner and the second named author have provided in \cite{BSZ} a universal property for the symmetric monoidal $\infty$-category $\OC$ of \emph{open-closed 2-dimensional oriented cobordisms}: while the symmetric monoidal $\infty$-category $\cO$ of \emph{open 2-dimensional oriented cobordisms} carries the universal $\E_1$-Frobenius algebra $\cA$, its extension $\OC$ is obtained by universally adjoining the factorization homology $\int_{M}\cA$ over all compact oriented 1-manifolds $M$. Most notably, taking $M=S^1$, we are also universally adjoining the Hochschild homology $\int_{S^1}\cA$ of $\cA$.

\begin{remark}
\label{rem:esquare}
The first named author constructs in \cite[Remark 1.5]{Bianchi:stringtopology} a symmetric monoidal functor that we denote here by $\epsilon\colon\OC\to\GrCob$. Let
$\cO\subset\OC$ denote the symmetric monoidal $\infty$-subcategory of \emph{open 2-dimensional oriented cobordisms}, i.e.~ the full $\infty$-subcategory spanned by disjoint unions of disks.
We then have that $\epsilon$ restricts to a functor $\epsilon\colon\cO\to\Gr$, so we obtain a commutative square of symmetric monoidal $\infty$-categories as follows:
\[
\begin{tikzcd}[row sep=10pt]
\cO\ar[r,"\epsilon"]\ar[d,hook] & \Gr\ar[d,"\iota",hook]\\
\OC\ar[r,"\epsilon"]&\GrCob.
\end{tikzcd}
\]
In fact, using that both vertical functors are fully faithful and that a compact 1-manifold is a disjoint union of disks if and only if it is homotopy equivalent to a finite set, we obtain that this square is even a pullback square. In light of Theorem A and the results of \cite{BarkanSteinebrunner,BSZ}, we can interpret this square via universal properties: the top and bottom horizontal maps are induced by the observations that any $\Ei$-Frobenius algebra has an underlying $\E_1$-Frobenius algebra, and that the existence of factorization homology over arbitrary spaces implies existence of factorization homology over compact 1-manifolds. The vertical maps  record respectively the $\E_1$ and $\E_\infty$-Frobenius algebras of which we are taking factorization homologies. 
\end{remark}

\begin{remark}
Precomposition by the functor $\epsilon\colon\OC\to\GrCob$ from Remark \ref{rem:esquare} yields, for any $(\V,\cA)\in\CFrobext$, a symmetric monoidal functor
\[
\int_{(-)}\cA\circ \epsilon\colon\OC\to\V,
\]
 also known as an \emph{open-closed 2-dimensional topological field theory}.
In particular, it evaluates to  $\cA(*)$ at the disk, and to the Hochschild homology $HH(\cA):=\int_{S^1}\cA$ at the circle $S^1$.
The homology of the morphism spaces of $\OC$, which are equivalent to suitable moduli spaces of Riemann surfaces with boundary, provides higher string operations relating the underlying object and the Hochschild homology of an $\Ei$-Frobenius algebra considered as an $\E_1$-Frobenius algebra. 
Open-closed 2-dimensional field theory of this kind have been extensively studied in the literature; a list of references includes \cite{TradlerZeinalian,TradlerZeinalian2,CostelloOC,CostelloGauge,CostelloTradlerZeinalian,Kaufmann1,Kaufmann2,Kaufmann3,lauda2007state,lauda2008open,WahlWesterland,wahl2016universal,muller2024categorified,BSZ}. 
Our Theorem \ref{thm:A} and Corollary \ref{cor:A} thus serve as a refinement of the cited literature, and we will make the connection even more precise in Theorem \ref{thm:C} below. 
\end{remark}
\begin{remark}
Of particular interest is the example where the $\Ei$-Frobenius algebra $\cA$ is the cochain complex $C^*(M;R)$, where $R$ is a discrete commutative ring and $M$ is an $R$-oriented manifold. On the one hand the $R$-orientation can be used to upgrade $C^*(M;R)$ to an $\Ei$-Frobenius algebra in the ``Picard-twisted'' symmetric monoidal $\infty$-category $\pazocal{D}(R)/\mathrm{Pic}(\pazocal{D}(R))$ of chain complexes over $R$: this is defined in \cite{BarkanSteinebrunner} as the quotient of $\pazocal{D}(R)$ by the action of $\mathrm{Pic}(\pazocal{D}(R))$ given by tensor product; see also \cite[Definition 5.3]{Bianchi:stringtopology} for an alternative, 2-categorical perspective.
On the other hand $\int_{S^1}\cA$ is equivalent, when $M$ is simply connected, to the cochain complex $C^*(\cL M;R)$ of the free loop space of $M$. In this case our construction provides string operations relating the cochain complexes $C^*(M;R)$ and $C^*(\cL M;R)$; these string operations are given by the homology of morphism spaces of $\GrCob$, and viaa restriction, by the homology of the morphism spaces of $\OC$. We stress however that if $M$ is not simply connected then in general $C^*(\cL M;R)$ is not equivalent to $\int_{S^1}C^*(M;R)$; in this case our results are not able to recover those from \cite{Bianchi:stringtopology}, which apply to arbitrary $R$-oriented manifolds.
\end{remark}

A notable application of Corollary \ref{cor:A} through the Yoneda embedding $\yo_\Gr\colon\Gr\to\Psh(\Gr)$, which is a symmetric monoidal functor if we equip $\Psh(\Gr)$ with the symmetric monoidal structure given by the Day convolution. In fact, Day convolution can be characterized as the essentially unique symmetric monoidal structure on $\Psh(\Gr)$ preserving colimits separately in each variable, such that $\yo_\Gr$ is enhanced to a symmetric monoidal functor \cite[Corollary 4.8.1.12]{HA}. Letting $\V=\Psh(\Gr)$ in Corollary \ref{cor:A}, we obtain a symmetric monoidal functor $\int_{(-)}\yo_\Gr\colon\GrCob\to\Psh(\Gr)$, whose underlying functor is the left Kan extension $\iota_!\yo_\Gr$. We will prove the following theorem that provides an alternative characterization of $\GrCob$.
\begin{athm}
\label{thm:B}
The left Kan extension $\iota_!\yo_\Gr\colon\GrCob\to\Psh(\Gr)$ is fully faithful and agrees with the composite symmetric monoidal functor
\[
\GrCob\xrightarrow{\yo_{\GrCob}}\Psh(\GrCob)\xrightarrow{\iota^*} \Psh(\Gr).
\]
Either functor identifies $\GrCob$ as the full symmetric monoidal $\infty$-subcategory of $\Psh(\Gr)$ spanned by the colimits $\int_X \yo_\Gr\simeq\colim_{\Fin_{/X}}\yo_\Gr$ for all spaces $X\in\Ss$.
\end{athm}
In a similar spirit, Barkan, Steinebrunner and the second named author identify $\OC$ as the full $\infty$-subcategory of $\Psh(\cO)$ spanned by colimits $\int_M\yo_{\cO}$ for all compact 1-manifolds $M$ \cite{BSZ}.

We conclude the article by identifying the spaces of universal ``natural operations'' between factorization homologies of $\Ei$-Frobenius algebras. The notion of natural operations on ``structured'' $\E_1$-algebras in $\pazocal{D}(\mathbb{Z})$ was introduced in \cite{wahl2016universal} as natural transformations of the Hochschild homology functor. 
In \cite[Section 3.3]{BSZ}, this notion was lifted to the level of spaces and $\OC$ was identified as the $\infty$-category of universal natural operations on the Hochschild homology of $\E_1$-Frobenius algebras in $\Ss$. In order to uniformly treat natural operations for both structured $\Ei$-algebras and structured $\E_1$-algebras, we shall introduce in Definition \ref{defn:Nat} the general notion of ``context for natural operations'', using which we will introduce the symmetric monoidal $\infty$-categories $\Nat_{\Disk_1}[\cO,\Mfld_1]$ and $\Nat_\Fin[\Gr,\Ss]$ of natural operations on the factorization homologies of $\E_1$ and $\Ei$-Frobenius algebras.
Our main theorem about natural operations  is the following, see Theorem \ref{thm:formaloperations} for the complete statement.
\begin{athm}\label{thm:C}
There is a commutative square of symmetric monoidal $\infty$-categories, with the horizontal arrows being equivalences:
\[
\begin{tikzcd}[row sep=10pt]
\OC\ar[d,"\epsilon"]\ar[r,"\simeq"]&\Nat_{\Disk_1}[\cO,\Mfld_1]\ar[d,"{\Nat[\epsilon]}"]\\
\GrCob\ar[r,"\simeq"]&\Nat_{\Fin}[\Gr,\Ss].
\end{tikzcd}
\]
\end{athm}
Since the factorization homology over $S^1$ of an $\Ei$-algebra agrees with the Hochschild homology of the underlying $\E_1$-algebra, Theorem \ref{thm:C} says that natural operations on the latter factors through natural operations on the former.

\textit{Acknowledgments.} We would like to thank Jan Steinebrunner for helpful conversations related to this project and detailed feedback to a draft of this article. This project can moreover be regarded as an analogue of \cite{BSZ} with $\GrCob$ in place of $\OC$ (and with $\Gr$ in place of $\cO$); we in particular acknowledge that the main ideas leading to Theorems \ref{thm:A} and \ref{thm:B} were adapted from there, and were also originally contributed by Jan Steinebrunner.

A.B. was supported by the Max Planck Institute for Mathematics in Bonn and by the University of Bologna. A.Z. was supported by the European Unions via a Marie
Curie postdoctoral fellowship (project 101150469) and the Centre for Geometry and Topology (DNRF151). We would also like to thank GeoTop  and the University of Copenhagen for the hospitality, with reference to a research visit of A.B. during which this project was born.
\section{Recollections}
We start by recalling the constructions and basic properties of the main players $\Gr$ and $\GrCob$, and their connection to the 2-dimensional cobordism categories $\cO$ and $\OC$.
\begin{definition}
We denote by $\Gr\subset\Cospan(\Ss)$ the $\infty$-subcategory of the $\infty$-category of cospans of spaces spanned by the following objects and morphisms:
\begin{itemize}
    \item objects in $\Gr$ are spaces equivalent to some finite set;
    \item for $A,B\in\Fin$, a morphism in $\Gr$ from $A$ to $B$ is a cospans of spaces of the form $A\to G\ot B$, where $G$ is equivalent to a finite cell complex of dimension at most 1.
\end{itemize}
We refer to $\Gr$ as the $\infty$-category of graph cobordisms between finite sets. Disjoint union of spaces makes $\Gr$ into a symmetric monoidal $\infty$-subcategory of $\Cospan(\Ss)$.
\end{definition}
We remark that $\Gr$ contains $\Fin$, the symmetric monoidal category of finite sets, as a symmetric monoidal $\infty$-category.
A recent result of Barkan--Steinebrunner \cite{BarkanSteinebrunner} (see also \cite{Bianchi:graphcobset}) provides a universal property for $\Gr$.
\begin{definition}
\label{defn:EiFrobenius}
Let $\V$ be a symmetric monoidal $\infty$-category. An \emph{$\Ei$-Frobenius algebra} in $\V$ is a pair $(\cA,\lambda)$ consisting of
\begin{itemize}
    \item an $\Ei$-algebra $\cA$ in $\V$, which we regard as a symmetric monoidal functor $\cA\colon\Fin\to\V$;
    \item a morphism $\lambda\colon\cA(*)\to 1_\V$ in $\V$,
\end{itemize}
such that the composition $\cA(*)\otimes\cA(*)\xrightarrow{\mu}\cA(*)\xrightarrow{\lambda}1_\V$ exhibits $\cA(*)$ as its own dual in $\V$. Here the ``multiplication'' $\mu$ is the image along $\cA$ of the morphism $*\sqcup *\to *$ in $\Fin$.

We denote by $\CFrob\colon\CAlg(\cat)\to\Ss$ the functor associating with $\V\in\CAlg(\cat)$ the mo\-duli space $\CFrob(\V)$ of $\Ei$-Frobenius algebras in $\V$.
\end{definition}
\begin{theorem}[\cite{BarkanSteinebrunner}]
\label{thm:BS}
The functor $\CFrob$ is corepresented by $\Gr$, i.e.~there is an equivalence $\CFrob\simeq\Fun^{\otimes}(\Gr,-)$ of functors $\CAlg(\cat)\to\Ss$.
\end{theorem}
We also briefly recall the definition of $\GrCob$ from \cite{Bianchi:stringtopology}.
\begin{definition}
Consider the symmetric monoidal $\infty$-category $\Cospan(\Fun([1],\Ss))$ of cospans of arrows of spaces.
We denote by $\Gr_\Ss\subset\Cospan(\Fun([1],\Ss))$ the symmetric monoidal $\infty$-subcatego\-ry of  spanned by objects of the form $A\to X$ with $A\in\Fin$, and by morphisms of the form
\[
\begin{tikzcd}[row sep=10pt]
A\ar[r]\ar[d]&G\ar[d]&B\ar[l]\ar[d]\ar[dl,phantom,"\urcorner"very near end]\\
X\ar[r]&W&Y\ar[l]
\end{tikzcd}
\]
such that the right square is a pushout square, and the top row gives a morphism in $\Gr$. We denote by $D_1,D_0\colon\Gr_\Ss\to\Cospan(\Ss)$ the functors sending $(A\to X)$ to $ A,X$ respectively. Note that $D_1$ factors through $\Gr$.

We denote by $\GrCob$ the (symmetric monoidal) localisation of $\Gr_\Ss$ at ``idle morphisms'', i.e.~at morphisms sent to equivalences along $D_0$; and we denote by abuse of notation also by $D_0\colon\GrCob\to\Cospan(\Ss)$ the functor tautologically induced on localisation.
\end{definition}
\begin{notation}
We denote by $-\otimes-$ the monoidal product of $\Gr$ and $\GrCob$ in order to distinguish them from the monoidal products of $\Fin$ and $\Ss$, denoted $-\sqcup-$ instead.
\end{notation}
By \cite[Corollary 7.9]{Bianchi:stringtopology}, the functor $D_0$ induces an equivalence on core groupoids; in particular it is fair to think of objects of $\GrCob$ as being spaces. By \cite[Corollary 7.11]{Bianchi:stringtopology}, the functor $D_0$ restricts to an equivalence between the full $\infty$-subcategory of $\GrCob$ spanned by finite sets, and $\Gr\subset\Cospan(\Ss)$. This provides us in particular with a symmetric monoidal full inclusion $\iota\colon\Gr\hookrightarrow\GrCob$. Combining this with the $\infty$-subcategory inclusion $\Ss\hookrightarrow\GrCob$ from \cite[Corollary 7.10]{Bianchi:stringtopology}, we obtain the following commutative square of symmetric monoidal $\infty$-categories whose vertical arrows are full inclusions, and whose horizontal arrows are wide inclusions:
\begin{equation}\label{sq:Fin->Gr->GrCob}
\begin{tikzcd}[row sep=10pt]
\Fin\ar[r]\ar[d]&\Gr\ar[d,"\iota"]\\
\Ss\ar[r]&\GrCob.
\end{tikzcd}
\end{equation}
In fact this is a pullback square in $\CAlg(\cat)$, but we will not need this remark.

On the other hand, recall the symmetric monoidal $\infty$-category $\Cob_2^{\partial}$ of oriented 2-dimensional cobordisms between oriented compact 1-manifolds with boundary, where cobordisms are oriented compact 2-manifolds with boundary and corners; see for instance \cite[Section 2]{BSZ} for the precise construction. The symmetric monoidal $\infty$-category $\OC$ of \emph{open-closed cobordisms} is defined as the wide $\infty$-subcategory of $\Cob_2^{\partial}$ spanned by cobordisms $A\to W\ot B$ such that no  connected component of $W$ has its entire boundary contained in $B$. Equivalently, we only allow cobordisms that have handle dimension at most 1 relative to the outgoing boundary. The symmetric monoidal $\infty$-subcategory $\cO$ of \emph{open cobordisms} is then defined as the full $\infty$-subcategory of $\OC$ spanned by disjoint unions of disks. By taking pullback of the square (\ref{sq:Fin->Gr->GrCob}) along the functor $\epsilon\colon\OC\to\GrCob$ from Remark \ref{rem:esquare}, we obtain the following commutative cube, in which all vertical functors are full inclusions, all horizontal functors are wide inclusions, and each face is a pullback square in $\CAlg(\cat)$:
\begin{equation}\label{sq:cube}
\begin{tikzcd}[row sep=6pt, column sep=10pt]
\Disk_1\ar[rr]\ar[rd]\ar[dd]&&\cO\ar[dd]\ar[rd,"\epsilon"]\\
&\Fin\ar[rr]\ar[dd]&&\Gr\ar[dd]\\
\Mfld_1\ar[rd]\ar[rr]&&\OC\ar[rd,"\epsilon"]\\
&\Ss\ar[rr]&&\GrCob
\end{tikzcd}
\end{equation}
Here $\Mfld_1$ denotes the symmetric monoidal $\infty$-category of compact oriented 1-manifolds and orientation-preserving embeddings; this is equivalent to the wide $\infty$-subcategory of $\OC$ spanned by cobordisms $A\to W\ot B$ such that the inclusion of the outgoing boundary $B\to W$ is a homotopy equivalence. While $\Disk_1$ denotes the symmetric monoidal full $\infty$-subcategory of $\Mfld_1$ spanned by disjoint unions of disks. 
In particular we observe that $\Disk_1$ corepresents $\E_1$-algebras in symmetric monoidal $\infty$-categories, whereas $\cO$ corepresents $\E_1$-Frobenius algebras in symmetric monoidal $\infty$-categories by \cite{BarkanSteinebrunner}.

\section{Canonical extensions and the universal property of \texorpdfstring{$\GrCob$}{GrCob}}
In this section, we show that any extendable $\Ei$-Frobenius algebra admits a canonical symmetric monoidal extension to $\GrCob$. As a consequence, we deduce the universal property of $\GrCob$ stated in Theorem \ref{thm:A}, which says that $\GrCob$ is the initial extendable $\Ei$-Frobenius algebra. 
\subsection{Canonical extensions of \texorpdfstring{$\Ei$}{Einfty}-Frobenius algebras}
Given a symmetric monoidal $\infty$-category $\V$ and an $\Ei$-Frobenius algebra $\cA\colon\Gr\to\V$, it is natural to ask whether $\cA$ admits an extension to a ``graph field theory'', i.e.~a symmetric monoidal functor $F:\GrCob\to\V$ such that the restriction of $F$ to $\Gr$ is $\cA$.
A good candidate for this extension is the \emph{operadic left Kan extension} of $\cA$ along the inclusion $\iota:\Gr\to\GrCob$. We take this notion from \cite[Section 3.1]{HA}, and refer to \cite[Theorem 7.6]{BSZ} for an alternative account that specializes to the case of interest and provides in particular the necessary and sufficient conditions for the existence of an operadic left Kan extension.
\begin{definition}
\label{defn:operadiklKe}
An operadic left Kan extension of a lax symmetric monoidal functor $F:\C\to \V$ along a lax symmetric monoidal functor $i:\C\to\D$ is a pair $(G, \alpha)$ of a lax symmetric monoidal functor $G\colon \D \to \V$ and a symmetric monoidal natural transformation $\alpha\colon F \to G \circ i$ such that for all $v \in \V$, the transformation induced by $\alpha$ exhibits $G(-) \otimes v\colon \D \to \V$ as the plain left Kan extension of $F(-) \otimes v\colon \C \to \V$ along $i$. In particular, for $v$ the monoidal unit, the underlying functor of $G$ is exhibited as the
left Kan extension of $F$ along $i$. 
\end{definition}
\begin{notation}
In the notation of Definition \ref{defn:operadiklKe}, we usually denote by $i_!^\otimes F\simeq G\colon\D\to\V$ the operadic left Kan extension, which is essentially unique, if it exists. We reserve the notation $i_!F\colon \D\to\V$ for the (plain) left Kan extension of $F$ along $i$, whose definition does not use any (lax) symmetric monoidal structures.
\end{notation}
A necessary and sufficient condition to ensure that an operadic left Kan extension $i_!^\otimes F$ exists is that both of the following hold:
\begin{itemize}
\item the left Kan extension $i_!(F(-)\otimes v)$ exists for all $v\in\V$;
\item the assembly map $i_!(F(-)\otimes v)\to i_!F(-)\otimes v$ is a natural equivalence.
\end{itemize} 
\begin{remark}
In some situations, e.g. when $\V$ is cocomplete and its monoidal product preserves colimits separately in each variable, then \emph{all} lax symmetric monoidal functors $\C\to\V$ admit an operadic left Kan extension $\D\to\V$. In this case, the operadic left Kan extension provides a left adjoint to the restriction functor  $i^*:\Fun^{\otimes,\lax}(\D,\V)\to\Fun^{\otimes,\lax}(\C,\V)$: see \cite[Section 3.1]{HA}, or \cite[Theorem 1.1]{LKE} for an alternative exposition. In general, however, the notion of operadic left Kan extension is different.
\end{remark}
Since $\iota\colon\Gr\to\GrCob$ is fully faithful, the operadic left Kan extension $\iota_!\cA$ of a symmetric monoidal functor  $\cA\colon\Gr\to\V$, if it exists, restricts indeed to $\cA$ on $\Gr$. We will show in the following proposition that its lax symmetric monoidal structure is in fact a (strong) symmetric monoidal structure, and we give a simplified formula for its pointwise values. Then we will deduce Theorem \ref{thm:A} as an immediate corollary.

\begin{proposition}\label{prop:LKS-symmonoidal}
Let $(\V,\cA)\in\CFrobext$ as in Definition \ref{defn:CFrobext}, which implies that the operadic left Kan extension $i_!^\otimes \cA$ exists. Then the lax symmetric monoidal structure on $\iota^\otimes_!\cA$ is in fact a strong symmetric monoidal structure.
\end{proposition}

We start by establishing some computational inputs to Proposition \ref{prop:LKS-symmonoidal}.
\begin{lemma}\label{lem: factorizationhomologyinGrCob}
For any space $X$ the canonical map
\[\int_X \iota:=\colim(\Fin_{/X}\to\Fin\hookrightarrow\Gr\xrightarrow{\iota}\GrCob)\too X
\]
is an equivalence in $\GrCob$; in particular the colimit defining the source exists.
\end{lemma}
\begin{proof}
For $Y\in\GrCob$ we want to show that $\GrCob(X,Y)\xrightarrow{\simeq}\lim_{A\in(\Fin_{/X})^\op}\GrCob(A,Y)$ is an equivalence of spaces. Both spaces map to $\GrCob(\emptyset,Y)$, using that $\emptyset$ is initial in $\Ss$ and regarding $\Ss$ as a wide $\infty$-subcategory of $\GrCob$ as in \cite[Corollary 7.10]{Bianchi:stringtopology}; therefore it suffices to prove, for all $w\in\GrCob(\emptyset,Y)$, that the induced map between fibres at $w$ is an equivalence. The datum of $w$ is roughly that of a map of spaces $Y\to W$, plus some information about how $W$ is obtained from $Y$ by attaching a graph. By \cite[Remark 7.12]{Bianchi:stringtopology} we have an equivalence as follows, induced by the functor denoted ``$D_0$'' in loc.cit.: 
\[
\fib_w(\GrCob(X,Y)\to\GrCob(\emptyset,Y))\xrightarrow{\simeq}\map(X,W);
\]
We similarly have an equivalence as follows, suitably compatible with the one above:
\[
\begin{split}
\fib_w\left(\lim_{A\in(\Fin_{/X})^\op}\GrCob(A,Y)\to\GrCob(\emptyset,Y)\right)&\simeq\lim_{A\in(\Fin_{/X})^\op}\fib_w(\GrCob(A,Y)\to\GrCob(\emptyset,Y))\\
&\xrightarrow{\simeq}\lim_{A\in(\Fin_{/X)^\op}}\map(A,W)\simeq\map(X,W),
\end{split}
\]
where we use that $X\simeq\colim_{A\in\Fin_{/X}}A$ in spaces.
% Sketch of proof of last statement: 1) $\colim_{A\in\Fin_{/X}}A$ is the classifying space of the left fibration over $\Fin_{/X}$ corresponding to the source functor $\Fin_{/X}\to\Fin\to\Ss$; 2) the total category of this left fibration is $(\Fin_*)_{/X}$, i.e. finite pointed sets with a map to $X$; 3) The full subcategory of $(\Fin_*)_{/X}$ spanned by pointed singletons over $X$ is on the one hand a groupoid and as such equivalent to $X$; on the other hand its inclusion admits a right adjoint.
\end{proof}
\begin{corollary}\label{cor: factorizationhomologyinGrCob}
For any space $X$ and any finite set $A$ the canonical map
\[
\int_X(\iota\otimes A)\to X\otimes A
\]
is an equivalence in $\GrCob$; in particular the colimit defining the source exists.
\end{corollary}
\begin{proof}
We may factor the map as the following composite:
\[
\int_X(\iota\otimes A)\to\int_{X\sqcup A}\iota\to X\otimes A.
\]
The second arrow is an equivalence by Lemma \ref{lem: factorizationhomologyinGrCob}, and in particular the middle colimit exists.
The first arrow is an equivalence, and in particular this also implies that the left colimit exists, because the functor $-\sqcup A\colon\Fin_{/X}\to\Fin_{/X\sqcup A}$ is final.
\end{proof}

\begin{lemma}\label{lem:colimhom}
    The canonical map \[\colim_{B\in\Fin_{/X}}\GrCob(A,B)\to\GrCob(A,X)\] is an equivalence for any space $X\in\GrCob$ and finite set $A\in\Gr$.
\end{lemma}
\begin{proof}
The above canonical map can be factored as a sequence of equivalences
\begin{align*}
\colim_{B\in\Fin_{/X}}\GrCob(A,B)&\simeq\colim_{B\in\Fin_{/X}}\big(\colim_{C\in\Fin_{/B}}\Gr_\Ss(\emptyset\to A,C\to B)\big)\\
&\xleftarrow{\simeq}\colim_{B\in\Fin_{/X}}\Gr_\Ss(\emptyset\to A,B=B)\\
&\xleftarrow{\simeq}\colim_{B\in\Fin_{/X}}\Gr_\Ss(A=A,B=B)\\
&\xrightarrow{\simeq}\colim_{B\in\Fin_{/X}}\Gr_\Ss(A=A,B\to X)\\
&\xrightarrow{\simeq}\colim_{B\in\Fin_{/X}}\Gr_\Ss(\emptyset\to A,B\to X)\simeq\GrCob(A,X).
\end{align*}
In the above sequence of equivalences we make use of the $\infty$-category $\Gr_\Ss$ from \cite{Bianchi:stringtopology}, of which $\GrCob$ is a localisation; we invite the reader to see loc.cit. for further details.
The first and last equivalence follow from the formula for morphism spaces in $\GrCob$ from \cite[Corollary 7.8]{Bianchi:stringtopology}; the second from the fact that, since $B$ is a finite set, the identity of $B$ is terminal in $\Fin_{/B}$; the third and the fourth equivalences follow from the equivalences of morphism spaces 
\[
\Gr_\Ss(\emptyset\to A,B=B)\xleftarrow{\simeq}\Gr_\Ss(A=A,B=B)\xrightarrow{\simeq}\Gr_\Ss(A=A,B\to X),
\]
induced by precomposition by $(\emptyset\to A)\to(A=A)$ and postcomposition by $(B=B)\to(B\to X)$, respectively; it is immediate to check from the definition of $\Gr_\Ss$ that these are indeed equivalences; finally, the fifth equivalence, induced by precomposition by $(\emptyset\to A)\to(A=A)$, follows from \cite[Proposition 7.4]{Bianchi:stringtopology}, i.e.~the fact that the localisation $\GrCob$ of $\Gr_\Ss$ can be computed by left calculi of fractions.
\end{proof}

\begin{corollary}
\label{cor:final}
Let $X$ be a space. Then the following functor is final:
\[
\Phi:\Fin_{/X}=\Fin\times_{\Ss}\Ss_{/X}\to \Gr\times_{\GrCob}\GrCob_{/X}.
\]
\end{corollary}
In defining $\Phi$ we are using that $\Ss$ is a wide $\infty$-subcategory of $\GrCob$, and we are identifying $\Fin$ with the intersection $\Gr\cap\Ss$ of $\infty$-subcategories of $\GrCob$. We observe that $\Phi$ is a full $\infty$-subcategory inclusion, since the wide $\infty$-subcategory $\Ss\subset\GrCob$ satisfies the 2-of-3 property. Namely, $\Ss$ can be characterised as the wide $\infty$-subcategory of $\GrCob$ spanned by graph cobordisms $X\to W\xleftarrow{\simeq}Y$ whose second arrow is an equivalence \cite[Corollary 7.10]{Bianchi:stringtopology}.
\begin{proof}
    The source of $\Phi$ is the unstraightening of the restricted Yoneda functor $\yo_{\Gr}^\circ:\GrCob\to\Psh(\GrCob)\to\Psh(\Gr)$ at $X$. Therefore, by \cite[Lemma 7.1]{BSZ}, to show that $\Phi$ is final, it suffices to show that the canonical map $\colim_{B\in\Fin_{/X}}\yo_{\Gr}(B)\to\yo_{\Gr}^\circ(X)$ is an equivalence. This is indeed true by Lemma \ref{lem:colimhom} above.
\end{proof}

Now we are ready to prove Proposition \ref{prop:LKS-symmonoidal}

\begin{proof}[Proof of Proposition \ref{prop:LKS-symmonoidal}]
The conditions on $\V$ allows us to apply \cite[Theorem 7.6]{BSZ} to deduce that the operadic left Kan extension $\iota^{\otimes}_!\cA$ along the inclusion $\iota:\Gr\to\GrCob$ exits. To show that this lax symmetric monoidal functor is in fact symmetric monoidal, it suffices to check the condition in \cite[Lemma 7.7]{BSZ}, namely the canonical functor
\[
(\Gr\times_{\GrCob}\GrCob_{/X})\times (\Gr\times_{\GrCob}\GrCob_{/Y})\to\Gr\times_{\GrCob}\GrCob_{/X\otimes Y}
\]
is final for any spaces $X,Y\in\GrCob$. This follows from Corollary \ref{cor:final}, the fact that the inclusion $\Fin\to\GrCob$ is symmetric monoidal, and the fact that the canonical functor $\Fin_{/X}\times\Fin_{/Y}\to\Fin_{/X\sqcup Y}$ is an equivalence for any spaces $X,Y$.
\end{proof}

\subsection{Proof of Theorem \ref{thm:A}}
As an immediate consequence of Proposition \ref{prop:LKS-symmonoidal}, we deduce Theorem \ref{thm:A}, which says that $(\GrCob,\iota)$ is extendable and it is initial among all extendable $\Ei$-Frobenius algebras.

\begin{proof}[Proof of Theorem \ref{thm:A}]
We first check that $(\GrCob,\iota)$ is extendable, i.e. it satisfies the conditions on $(\V,\cA)$ in Definition \ref{defn:CFrobext}.
Given Lemma \ref{lem: factorizationhomologyinGrCob}, it remains to show that $\int_X(\iota\otimes Y)\xrightarrow{\simeq}(\int_X\iota)\otimes Y$ for all spaces $X$ and $Y$. Using Lemma \ref{lem: factorizationhomologyinGrCob} and Corollary \ref{cor: factorizationhomologyinGrCob}, the fact that $\iota$ attains finite sets as values, the equivalence $\Fin_{/X\sqcup Y}\simeq\Fin_{/X}\times\Fin_{/Y}$, and the fact that the inclusion $\Ss\hookrightarrow\GrCob$ is symmetric monoidal, we obtain the desired chain of equivalences
\[
\int_X(\iota\otimes Y)\simeq\int_X\left(\iota\otimes\int_Y\iota\right)\simeq\int_{X\sqcup Y}\iota\simeq X\otimes Y\simeq \left(\int_X\iota\right)\otimes Y.
\]
The fact that $(\GrCob,\iota)\in\CFrobext$ is an initial object now follows from observing that a morphism $F\colon(\GrCob,\iota)\to(\V,\cA)$ in $\CFrobext$ is by Corollary \ref{cor:final} precisely the datum of a strong symmetric monoidal operadic left Kan extension of $\cA$ along $\iota$; the hypotheses on $(\V,\cA)$ now imply that the operadic left Kan extension exists, and Proposition \ref{prop:LKS-symmonoidal} implies that it is also strong symmetric monoidal.
\end{proof}
\begin{remark}\label{rem:canmapeq}
As already mentioned, the underlying functor of the operadic left Kan extension $\iota_!^\otimes\cA\colon\GrCob\to\V$ is the plain left Kan extension $\iota_!\cA\colon\GrCob\to\V$. Thanks to Corollary \ref{cor:final}, this can be characterised as the essentially unique functor $F$ equipped with an equivalence $\iota^*F\simeq \cA$ such that the canonical map 
\begin{equation}\label{eq:pointwise}
\int_X \cA =\colim(\Fin_{/X}\to\Fin\hookrightarrow\Gr\xrightarrow{\cA}\V)\to F(X)
\end{equation}
is an equivalence in $\V$ for any space $X$. 
\end{remark}

\begin{proof}[Proof of Corollary \ref{cor:A}]
The essentially unique symmetric monoidal extension $\int_{(-)}\cA$ is the operadic left Kan extension $\iota_!^\otimes\cA$, which by Proposition \ref{prop:LKS-symmonoidal} is symmetric monoidal. The underlying functor of $\iota_!^\otimes\cA$ is $\iota_!\cA$, and by Corollary \ref{cor:final} for $X\in\GrCob$ we indeed have an equivalence $\int_X\cA\xrightarrow{\simeq}\iota_!\cA(X)$.
\end{proof}

\begin{remark}
Proposition \ref{prop:LKS-symmonoidal} and hence in Theorem \ref{thm:A} still holds if we slightly relax the assumption on $(\V,\cA)\in\CFrobext$ by replacing the second point in Definition \ref{defn:CFrobext} with the following:
for all $X,Y\in\Ss$ the assembly map 
 \[
 \int_{X\sqcup Y}\cA\simeq\colim_{\Fin_{/X}\times\Fin_{/Y}}\cA\otimes\cA\to(\colim_{\Fin_{/X}}\cA)\otimes(\colim_{\Fin_{/Y}}\cA)\simeq\int_X\cA\otimes\int_Y\cA.\]
 is an equivalence.
First we observe that the left Kan extension $\iota_!\cA:\Gr\to\V$  has essential image in the full $\infty$-subcategory $\V'\subseteq\V$ spanned by the objects $\int_X\cA$ for all spaces $X\in\Ss$. The assumption on $\V$ implies indeed that $\V'$ is closed under the monoidal product on $\V$. Furthermore, considering $\cA$ as an $\Ei$-Frobenius algebra in $\V'$, the pair $(\V',\cA)$ satisfies the condition of Proposition \ref{prop:LKS-symmonoidal} and hence we obtain a symmetric monoidal extension $\GrCob\to\V'$. Postcomposing with the symmetric monoidal inclusion $\V'\to\V$ yields a symmetric monoidal extension $\GrCob\to\V$ of $\cA$. However, this construction is in general not the \emph{operadic} left Kan extension of $\cA$, which may not even exist.
\end{remark}

\section{Denseness and Natural operations}
In this section we prove Theorem \ref{thm:B}, which says that $\Gr$ is dense in $\GrCob$ and characterizes $\GrCob$ as the extension of $\Gr$ obtained by formally adjoining all colimits
$\int_{X}\yo_{\Gr}(\iota)$ to $\Gr$. 
As an application, we determine the space of universal natural operations on the factorization homologies of $\Ei$-Frobenius algebras, which we then compare with the natural operations on the factorization homology of $\E_1$-Frobenius algebras from \cite{BSZ}. 

\subsection{Denseness of \texorpdfstring{$\iota:\Gr\to\GrCob$}{iota:Gr->GrCob}}
We start by recalling the notion of denseness.
\begin{definition}\cite[03VD]{Kerodon}
Let $i:\D\to\C$ be the inclusion of a full $\infty$-subcategory. We say that $\D$ is \textit{dense} in $\C$ if for every $X\in\C$, the following is a colimit diagram 
\[
(\D\times_\C \C_{/X})^\rhd \to \C_{/X} \to \C.
\]
\end{definition}
The following lemma provides two alternative characterizations of denseness.
\begin{lemma}
\label{lem:denseness}
    A full inclusion $i\colon \D\to\C$ is dense if and only if the following two equivalent conditions are satisfied:
\begin{enumerate}
\item \cite[03VE]{Kerodon} The (plain) left Kan extension of $i$ along $i$ is the identity of $\C$. In particular, the left Kan extension exists.
\item  \cite[03VG]{Kerodon}
The following restricted Yoneda embedding is fully faithful:
\[  \yo^\circ_{\D}: \C\xlongrightarrow{\yo_{\C}}\Psh(\C)\xlongrightarrow{i^*}\Psh(\D).
\]
\end{enumerate}
\end{lemma}

To prove Theorem \ref{thm:B}, we will use the first characterization to check denseness of $\Gr$ in $\GrCob$ and the second one to characterize $\GrCob$ as formally adjoining certain colimits to $\Gr$.

\begin{proof}[Proof of Theorem \ref{thm:B}]
By Lemma \ref{lem:denseness}.(1), it suffices to check that the left Kan extension $\iota_!\iota$ of the inclusion $\iota:\Gr\to\GrCob$ along $\iota$ is the identity on $\GrCob$.  It follows from  Lemma \ref{lem: factorizationhomologyinGrCob} and Corollary \ref{cor:final} that this left Kan extension indeed exists.  Since the identity $\id:\GrCob\to\GrCob$ restricts to the inclusion $\iota:\Gr\to\GrCob$, there is a natural transformation $\iota_!\iota\Rightarrow\id$. Furthermore, for any $X\in\GrCob$, this natural transformation evaluates to the canonical map
\[
\iota_!\iota(X)=\colim_{B\in\Gr\times_{\GrCob}\GrCob_{/X}}B=\colim_{B\in\Fin\times_{\Ss}\Ss_{/X}}B\xrightarrow{\simeq} X.
    \]
 which again by Lemma \ref{lem: factorizationhomologyinGrCob} and Corollary \ref{cor:final} is an equivalence; we thus have a natural isomorphism $\iota_!\iota\simeq\id$ of endofunctors of $\GrCob$ as desired.

  As a result, the restricted Yoneda embedding $\yo^\circ_{\Gr}:\GrCob\to\Psh(\Gr)$ is fully faithful by Lemma \ref{lem:denseness}(2). Since it restricts to $\yo_{\Gr}$ on the full $\infty$-subcategory $\Gr$, we obtain a natural transformation $\iota_!\yo_{\Gr}\Rightarrow \yo^{\circ}_{\Gr}$. This evaluates at any space $X\in\GrCob$ to the canonical map $\colim_{B
\in\Fin_{/X}}\yo^\circ_{\Gr}(B)\to \yo_{\GrCob}^\circ(X)$, which is an equivalence  by Lemma \ref{lem:colimhom}; we thus have a natural isomorphism $\iota_!\yo_\Gr\simeq\yo^\circ_\Gr$ as desired. The identification of the essential image of $\iota_!\yo_\Gr$ follows from Corollary \ref{cor:final}.
\end{proof}

\subsection{Natural operations on factorization homology of \texorpdfstring{$\Ei$}{Einfty}-Frobenius algebras}

Now we turn to the study of natural operations between factorization homologies of an $\Ei$-Frobenius algebra. 

\begin{definition}
\label{defn:Nat}
A \emph{context for natural operations} is a span of symmetric monoidal $\infty$-catego\-ries $\cE\ot\cF\to\cG$, with $\cF$ being a small $\infty$-category.
Given a context for natural operations as above and a presentably symmetric monoidal $\infty$-category $\V$, for any object $X\in\cG$ we let $I^{\cE}_X$ be the colimit preserving and $\V$-linear functor given by the composite
\[
I^{\cE}_X\colon\Fun(\cE,\V)\to\Fun(\cF,\V)\to\Fun(\cF\times_\cG\cG_{/X},\V)\xrightarrow{\colim}\V.
\] We define the $\infty$-category of natural operations of $\cF$-algebras in $\V$ with structures in $\cE$ and extensions to $\cG$
\[
\Nat_\cF[\cE,\cG;\V]\subseteq\Fun^{L,\V}(\Fun(\cE,\V),\V)\simeq\Psh(\cE;\V)
\]
as the full $\infty$-subcategory spanned by  $I^{\cE}_X$ 
as $X$ ranges over the objects of $\cG$.
When $\V=\Ss$ we omit it from the notation.
\end{definition}
\begin{example}
\label{ex:contexts}
The main context for natural operations we will consider are the spans $\cO\ot\Disk_1\to\Mfld_1$ and $\Gr\ot\Fin\to\Ss$. The restrictions of the functor $\epsilon:\OC\to\GrCob$ given in (\ref{sq:cube}) provide a comparison map between these two contexts for natural operations. 
\end{example}
The notion of natural operations was introduced by Wahl in \cite{wahl2016universal} in the setting $\V=D(\Z)$, $\cF\simeq\Disk_1$ and $\cG=\Mfld_1$, with arbitrary choice of $\cE$. In \cite{BSZ}, the situation was ``lifted'' to the setting $\V=\Ss$, the initial presentably symmetric monoidal $\infty$-category, where the focus is on $\cE=\cO$. In favourable cases including the two considered in Example \ref{ex:contexts} and Theorem \ref{thm:C}, which we will demonstrate, the $\infty$-category $\Nat_\cF[\cE,\cG;\V]$ happens to be a \emph{symmetric monoidal} $\infty$-subcategory of $\Psh(\cE;\V)$, where the latter is endowed with the Day convolution symmetric monoidal structure.

First we establish a criterion for when a map between contexts of natural operations induces a comparison map between the $\infty$-categories of natural operations.
\begin{lemma}\label{lem:compareoperations}
Suppose that we have a map of spans of $\infty$-categories
\[
\begin{tikzcd}[row sep=10pt]
\cE\ar[d,"e"]&\cF\ar[l]\ar[r]\ar[d,"f"]&\cG\ar[d,"g"]\\
\cE'&\cF'\ar[r]\ar[l]&\cG',
\end{tikzcd}
\]
and that the induced functor $\cF\times_{\cG}\cG_{/X}\to\cF'\times_{\cG'}\cG'_{/g(X)}$ is final for any $X\in\cG$.
Then there is a commutative diagram as follows:
\[
\begin{tikzcd}[row sep=10pt]
\Nat_{\cF}[\cE,\cG;\V]\ar[r,phantom,"\subseteq"]\ar[d,"{\Nat[e]}"]&\Fun^L(\Fun(\cE,\V),\V)\ar[r,phantom,"\simeq"]\ar[d,"-\circ e^*"]&\Psh(\cE;\V)\ar[d,"e_!"]\\
\Nat_{\cF'}[\cE',\cG';\V]\ar[r,phantom,"\subseteq"]&\Fun^L(\Fun(\cE',\V),\V)\ar[r,phantom,"\simeq"]&\Psh(\cE';\V).
 \end{tikzcd}
 \]
\end{lemma}
\begin{proof}
For any $X\in\cG$, consider the following diagram:
\[
\begin{tikzcd}[row sep=10pt]
\Fun(\cE',\V)\ar[d,"e^*"]\ar[r]&\Fun(\cF'\times_{\cG'}\cG'_{/g(X)},\V)\ar[d, ]\ar[dr]\\
\Fun(\cE,\V)\ar[r]&\Fun(\cF\times_{\cG}\cG_{/X},\V)\ar[r]&\V
\end{tikzcd}.
\]
In the right triangle, the two arrows to $\V$ are taking colimits over the slice categories, so the right triangle commutes by the finality assumption. The left square is commutative by construction. 
Note that the composites along the top and bottom are respectively $I^{\cE'}_{g(X)}$ and $I^{\cE}_X$, so we obtain an equivalence $I^{\cE'}_{g(X)}\simeq I^{\cE}_X\circ e^*$, which shows that the restriction of $-\circ e^*$ to $\Nat_\cF[\cE,\cG;\V]$ factors through $\Nat_{\cF'}[\cE',\cG';\V]$ as desired.
\end{proof}
In order to show that the map of contexts for natural operations from Example \ref{ex:contexts} satisfies the assumption in Lemma \ref{lem:compareoperations}, we will need the 
following well-known result. We were unable to find a precise reference in the literature, so we provide a proof here.
In particular, this implies that the factorization homology of an $\Ei$-algebra $\cA$ over $S^1$ agrees with its Hochschild homology, where we consider $\cA$ as an $\E_1$-algebra \cite[IV.2.2]{NS2018}\cite[Proposition 5.7]{ayala2020handbook}.
\begin{lemma} \label{lem:final Disk/S^1->Fin}
For any compact 1-manifold $M$, the functor 
\[
\Disk_{/M}:=\Disk_1\times_{\Mfld_1}(\Mfld_1)_{/M}\to \Fin_{/M}
\]
induced by taking path components of disks and regarding 1-manifolds as spaces is final.
    
\end{lemma}
\begin{proof}
Since the canonical maps $\Disk_{/M}\times \Disk_{/N}\to\Disk_{/M\sqcup N}$ and $\Fin_{/M}\times \Fin_{/N}\to\Fin_{/M\sqcup N}$ are equivalences for all compact 1-manifolds $M$ and $N$, it suffices to prove the lemma in the two cases $M=D^1$ and $M=S^1$. For $M=D^1$, this is clear since both the source and the target have a terminal object given by $D^1$. For $M=S^1$,
we identify $\Fin_{/S^1}$ with the 1-category having the following objects and morphisms:
\begin{itemize}
\item An object is a set $A$ with free $\Z$-action admitting finitely many orbits;
\item A morphism is a $\Z$-equivariant map.
\end{itemize}
Similarly, $\Disk_{/S^1}$ is equivalent to the paracyclic category $\Lambda_\infty$ \cite[Lemma 5.13 and 5.18]{BSZ}, which can be described as the following 1-category:
\begin{itemize}
\item  Objects are totally ordered sets $B$ with free $\Z$-action admitting finitely many orbits, such that the map $1.-\colon B\to B$ is order-preserving and satisfies $1.b>b$ for all $b\in B$;
\item A morphism is a $\Z$-equivariant and weakly increasing map.
\end{itemize}
To apply Quillen Theorem A, we want to show that for any object $A\in\Fin_{/S^1}$, regarded as a $\Z$-set as above, and the 
slice category $\C_A:=\Disk_{/S^1}\times_{\Fin_{/S^1}}(\Fin_{/S^1})_{A/}$ has contractible classifying space. We represent an object in $\C_A$ as a map of $\Z$-sets $A\to B$, where $B$ is endowed with a total order as above. Let $\C'_A$ denote the full subcategory of $\C_A$ spanned by surjective maps $A\to B$. Then the inclusion $\C'_A\hookrightarrow\C_A$ admits a right adjoint $R_A$, sending $A\to B$ to $A\twoheadrightarrow\Imm(A\to B)$, where the image $\Imm(A\to B)\subseteq B$ is endowed with the restricted total order. In particular $|\C_A|\simeq|\C'_A|$. We also observe that morphisms in $\C'_A$ are represented by \emph{surjective} $\Z$-equivariant and weakly increasing maps $B\twoheadrightarrow B'$ under $A$. 

We proceed by induction on the number of $\Z$-orbits of $A$. For $A=\emptyset$ we observe that $\C_\emptyset:=\Disk_{/S^1}\times_{\Fin_{/S^1}}(\Fin_{/S^1})_{\emptyset/}\simeq\Disk_{/S^1}$ has an initial object $\emptyset$, so $\C_\emptyset\simeq\C'_\emptyset\simeq*$ as categories. 

We next suppose that $|\C'_A|\simeq *$ for a given $A$, and we consider the $\Z$-set $A':=A\sqcup\Z$. The composite functor 
\[
\psi\colon\C'_{A'}\hookrightarrow\C_{A'}\xrightarrow{A\hookrightarrow A'}\C_A\xrightarrow{R_A}\C'_A
\]
is a cocartesian fibration. The fibre $\psi^{-1}(A\twoheadrightarrow B)$ of $\psi$ over an object $A\twoheadrightarrow B$ in $\C'_A$ is the poset described by the following objects and morphisms:
\begin{itemize}
    \item there are two types of objects:
    \begin{enumerate}
        \item the first type is given by surjections $A'=A\sqcup\Z\twoheadrightarrow B$ extending the given surjection $A\twoheadrightarrow B$;
        \item the second type is given by surjections $A'=A\sqcup\Z\twoheadrightarrow B\sqcup \Z$ given by the coproduct of the given  surjection $A\twoheadrightarrow B$ and the identity of $\Z$; these objects also carry the information of an extension to $B\sqcup\Z$ of the total order on $B$, satisfying the above requirements.
    \end{enumerate}
\item a morphism can only go from an object of type (2) to an object of type (1), and it is the unique map under $A'$, if it exists.
\end{itemize}
The poset $\psi^{-1}(A\twoheadrightarrow B)$ is readily identified with the poset of subintervals of the totally ordered set $B$ of one or two elements, and the partial ordering is giving by inclusion; in particular we have $|\psi^{-1}(A\twoheadrightarrow B)|\simeq*$. Using the inductive hypothesis $|\C'_A|\simeq*$, we conclude that also $|\C'_{A'}|\simeq*$.
\end{proof}
Now we are ready to prove Theorem \ref{thm:C}, which we state again using a larger commutative diagram.

\begin{theorem}\label{thm:formaloperations}

There is a commutative square of symmetric monoidal $\infty$-categories, with the left horizontal arrows being equivalences:
\[
\begin{tikzcd}[row sep=10pt]
\OC\ar[d,"\epsilon"]\ar[r,"\simeq"]&\Nat_{\Disk_1}[\cO,\Mfld_1] \ar[r,phantom,"\subseteq"]\ar[d,"{\Nat[\epsilon]}"]&\Psh(\cO)\ar[d,"\epsilon_!"]\\
\GrCob\ar[r,"\simeq"]&\Nat_{\Fin}[\Gr,\Ss] \ar[r,phantom,"\subseteq"]&\Psh(\Gr).
\end{tikzcd}
\]
\end{theorem}

\begin{proof}
The middle vertical map is well-defined by Lemma \ref{lem:compareoperations} and Lemma \ref{lem:final Disk/S^1->Fin}, so that the right square automatically commutes.
The top left horizontal arrows is an equivalence by \cite[Corollary 3.6]{BSZ}, where $M\in\OC$ and $I^{\cO}_M$ both correspond to the presheaf $\int_M \yo_{\cO}(\iota)$. The bottom left horizontal arrow is an equivalence by Theorem \ref{thm:B}: $\GrCob$ and $\Nat_{\Fin}[\Gr,\Ss]$ are both identified as full subcategories of $\Psh(\Gr)$ spanned by the same objects and hence equivalent, with $X\in\GrCob$ and $I^{\Gr}_X$ both corresponding to the presheaf $\int_X \yo_{\Gr}(\iota)$. 

The commutativity of the left square follows from the commutativity of the outer square, which follows from the naturality of the (restricted) Yoneda embedding.
\end{proof}

\begin{remark}
Let $\V$ be an arbitrary presentably symmetric monoidal $\infty$-category. On the one hand, one can use the essentially unique morphism $\Ss\to\V$ of presentably symmetric monoidal $\infty$-categories to transform $\epsilon\colon\OC\to\GrCob$, regarded as a morphism of $\Ss$-enriched symmetric monoidal $\infty$-categories, into a morphism $\epsilon^\V\colon\OC^\V\to\GrCob^\V$ of $\V$-enriched symmetric monoidal $\infty$-categories. On the other hand $\epsilon_!\colon\Psh(\cO;\V)\to\Psh(\Gr;\V)$ may be regarded as a morphism of $\V$-enriched symmetric monoidal $\infty$-categories, restricting to a $\V$-enriched symmetric monoidal functor 
\[
\Nat[\epsilon]^\V\colon\Nat_{\Disk_1}[\cO,\Mfld_1;\V]\to\Nat_\Fin[\Gr,\Ss;\V].
\]
The analogue to Theorem C identifies the $\V$-enriched symmetric monoidal functors $\epsilon^\V$ and $\Nat[\epsilon]^\V$, along with their sources and targets. The proof is identical and hence we omit the details.
\end{remark}

\bibliographystyle{alpha}
\bibliography{bib}

\end{document}